\documentclass{amsart}
\usepackage{amsmath,amssymb,amsthm,amsfonts}
\usepackage{mathrsfs}
\usepackage{amscd} 

\usepackage{bm}
\usepackage{hyperref}
\usepackage[normalem]{ulem}
\usepackage[T1]{fontenc}
\usepackage{lmodern}      
\usepackage[english]{babel}
\usepackage{microtype}               
\usepackage{graphicx}
\usepackage{enumitem}
\usepackage[curve]{xypic}
\usepackage[usenames,dvipsnames,x11names]{xcolor}
\usepackage{tikz}
\usepackage[all]{xy}
\usepackage{tikz-cd}
\usetikzlibrary{matrix}
\usetikzlibrary{arrows,decorations,shapes,shadows}
\usepackage{verbatim}
\usepackage{stmaryrd}
\usepackage{ocgx}

\newbox\mybox
\def\overtag#1#2#3{\setbox\mybox\hbox{$#1$}\hbox to
  0pt{\vbox to 0pt{\vglue-#3\vglue-\ht\mybox\hbox to \wd\mybox
      {\hss$\ss#2$\hss}\vss}\hss}\box\mybox}
\def\undertag#1#2#3{\setbox\mybox\hbox{$#1$}\hbox to 0pt{\vbox to
    0pt{\vglue#3\vglue\ht\mybox\hbox to \wd\mybox
      {\hss$\ss#2$\hss}\vss}\hss}\box\mybox}
\def\lefttag#1#2#3{\hbox to 0pt{\vbox to 0pt{\vglue -6pt\hbox to
      0pt{\hss$\ss#2$\hskip#3}\vss}}#1}
\def\righttag#1#2#3{\hbox to 0pt{\vbox to 0pt{\vglue -6pt\hbox to
      0pt{\hskip#3$\ss#2$\hss}\vss}}#1}
\let\ss\scriptstyle
\def\splicediag#1#2{\xymatrix@R=#1pt@C=#2pt@M=0pt@W=0pt@H=0pt}
\def\Dot{\lower.2pc\hbox to 2pt{\hss$\bullet$\hss}}
\def\Circ{\lower.2pc\hbox to 2pt{\hss$\circ$\hss}}
\def\Vdots{\raise5pt\hbox{$\vdots$}}
\newcommand\lineto{\ar@{-}}
\newcommand\dashto{\ar@{--}}
\newcommand\dotto{\ar@{.}}

\usepackage{xcolor}

\newcommand{\fract}[2]{\hbox{\leavevmode
  \kern.1em \raise .25ex \hbox{\the\scriptfont0 $#1$}\kern-.1em }\big/
  {\hbox{\kern-.15em \lower .5ex \hbox{\the\scriptfont0 $#2$}} }}

\renewcommand{\setminus}{\smallsetminus}

\newcommand\R{{\mathbb R}}
\newcommand\C{{\mathbb C}}

\newcommand\N{{\mathbb N}}
\newcommand\pa{{\mathfrak a}}

\newcommand\pb{{\mathfrak b}}

\newcommand\pc{{\mathfrak c}}

\newcommand{\Ker}{\operatorname{Ker}}
\renewcommand\O{{\mathcal O}}
\newcommand\lb{\llbracket}
\newcommand\rb{\rrbracket}
\renewcommand{\phi}{\varphi}
\renewcommand{\leq}{\leqslant}
\renewcommand{\geq}{\geqslant}

\newcommand\ub{{\mathbf{u}}}

\newcommand{\lgw}{\longrightarrow}
\newcommand{\lgm}{\longmapsto}

\newcommand{\ovl}{\overline}
\newcommand{\wdh}{\widehat}

\newcommand{\ini}{\operatorname{in}}
\newcommand{\wdt}{\widetilde}

\newcommand{\la}{\lambda}
\renewcommand{\O}{\mathcal{O}}
\renewcommand{\k}{\Bbbk}
\renewcommand{\a}{\alpha}
\renewcommand{\b}{\beta}
\newcommand{\e}{\varepsilon}

\newcommand{\K}{\mathbb{K}}
\newcommand{\g}{\gamma}

\newcommand{\ord}{\operatorname{ord}}

\newcommand{\D}{\Delta}
\newcommand{\NN}{\mathcal N}
\newcommand{\SA}{\mathcal{SA}}
\newcommand{\RR}{\mathcal R}
\newcommand{\Om}{\Omega}
\newcommand{\KK}{\mathcal K}
\newcommand{\LLL}{\mathcal L}

\newcommand{\x}{\mathbf{x}}

\newcommand{\w}{\mathbf{w}}

\newcommand{\nl}{\{\!\{}
\newcommand{\nr}{\}\!\}}

\newcommand{\Gr}{\operatorname{r}}
\newcommand{\Fr}{\operatorname{r}^{\mathcal{F}}}
\newcommand{\Ar}{\operatorname{r}^{\mathcal{A}}} 
\newcommand{\Tr}{\operatorname{r}^{\mathcal{W}}} 
\newcommand{\Wr}{\operatorname{r}^{\mathcal W}}

\newtheorem{theorem}{Theorem}[section]
\newtheorem{proposition}[theorem]{Proposition}
\newtheorem*{theorem*}{Theorem}

\newtheorem{corollary}[theorem]{Corollary}
\newtheorem{lemma}[theorem]{Lemma}
\newtheorem{claim}[theorem]{Claim}

\theoremstyle{definition}

\newtheorem*{amalgamation*}{Amalgamation}
\newtheorem{example}[theorem]{Example}
\newtheorem{remark}[theorem]{Remark}
\newtheorem{problem*}[theorem]{Problem}
\newtheorem*{remark*}{Remark}
\newtheorem{definition}[theorem]{Definition}

\usepackage{epigraph}

\begin{document}
\title{On the Nash points of subanalytic sets}
\author[A.~Belotto da Silva]{Andr\'e Belotto da Silva}
\author[O.~Curmi]{Octave Curmi}
\author[G.~Rond]{Guillaume Rond}
\address[A.~Belotto da Silva]{Université Paris Cité and Sorbonne Université, CNRS, IMJ-PRG, F-75013 Paris, France.}
\address[G.~Rond]{Universit\'e Aix-Marseille, Institut de Math\'ematiques de Marseille (UMR CNRS 7373), Centre de Math\'ematiques et Informatique, 39 rue F. Joliot Curie, 13013 Marseille, France.}
\address[O.~Curmi]{Alfréd Rényi Institute of Mathematics, Budapest, Hungary.}
\email[A.~Belotto da Silva]{andre.belotto@imj-prg.fr}
\email[O.~Curmi]{curmi@renyi.hu}
\email[G.~Rond]{guillaume.rond@univ-amu.fr}
\thanks{}



\begin{abstract}
Based on a recently developed rank Theorem for Eisenstein power series, we provide new proofs of the following two results of W. Paw\l ucki:\\
I) The non regular locus of a complex or real analytic map is an analytic set.\\
II) The set of semianalytic or Nash points of a subanalytic set $X$ is a subanalytic set, whose complement has codimension two in $X$.
\end{abstract}

\maketitle


\setlength{\epigraphwidth}{9.3truecm}
\epigraph{Algebra is the offer made by the devil to the mathematician. The
devil says: ``I will give you this powerful machine, it will answer any
question you like. All you need to do is give me your soul: give up
geometry and you will have this marvellous machine."}{Sir Michael Atiyah, (Collected works. Vol. 6.\\
Oxford Science Publications, 2004).}

\section{Introduction}\label{sec:Intro}

We provide new proofs of two fundamental results of analytic and subanalytic geometry due to Paw\l ucki \cite{Pawthesis, Paw}:

\smallskip
\noindent
 I) the non-regular (in the sense of Gabrielov) locus of a complex or real-analytic map $\Phi: M \lgw N$ is a proper analytic subset of $M$, see Theorem \ref{thm:PawI},

\smallskip 
 \noindent
  II) the set of semianalytic or Nash points of a subanalytic set $X$ is a subanalytic set, whose complement has dimension has dimension $\leq\dim(X)-2$, see Theorem \ref{thm:PawII}.
  
  \smallskip
  This last result answers a question asked by H. Hironaka and S. \L ojasiewicz independently \cite{FG1}. In spite of being considered as fundamental results of subanalytic geometry, their original proofs are considered to be very hard, as noted by \L ojasiewicz: ``Sans doute, parmi les faits établis en géométrie sous-analytique le théorème de Paw\l ucki [result II] est le plus difficile à prouver (la démonstration compte environ soixante dix pages!)", \cite[Page 1591]{LojCite}. The goal of this paper is to provide short alternative proofs of these results. We develop, furthermore, new algebraic methods to subanalytic geometry, notably related to Eisenstein power series, which we expect to be of independent interest. These methods should be useful in order to extend some of our results in the case of $p$-adic subanalytic sets, cf. \cite{DVdD}.

\bigskip

Let $\K = \C$ or $\R$ and denote by $\K \{x_1,\ldots,x_n\}$ the sub-ring of formal power series which are convergent, that is, $\K$-analytic. Given a ring homomorphism:
\[
\phi : \K\{ \x \}\lgw \K\{ \ub \}
\]
where $\x = (x_1,\ldots,x_n)$ and $\ub = (u_1,\ldots, u_m)$, we say that $\phi$ is a morphism convergent power series if $\phi(f) = f(\phi(\x))$ for every $f\in \K\{ \x \}$. We denote by $\wdh{\phi}$ its extension to the ring of formal power series and we define:
\begin{equation}
\begin{aligned}
\text{the Generic rank:} & &\Gr(\phi) &:= \mbox{rank}_{\mbox{Frac}(\K \{ \ub \} )}(\mbox{Jac}(\phi)),\\
\text{the Formal rank:}& &\Fr(\phi) &:= \dim\left(\frac{\K \lb \x \rb  }{\Ker(\wdh{\phi})}\right),\\
\text{and the analytic rank:}& &\Ar(\phi)&:=\dim\left(\frac{\K \{ \x \}}{\Ker(\phi)}\right),
\end{aligned}
\end{equation}
of $\phi$, where $\mbox{Jac}(\phi)$ stands for the matrix $[\partial_{u_i} \phi(x_j)]_{i,j}$. The morphism $\phi$ is said to be \emph{regular} (in the sense of Gabrielov) if $\Gr(\phi)=\Fr(\phi)$. We recall that Gabrielov's rank Theorem \cite{Ga1,To2,BCR} states that:
\[
 \Gr(\phi) = \Fr(\phi) \implies \Gr(\phi) = \Fr(\phi) = \Ar(\phi).
\]
Consider a $\K$-analytic map $\Phi : M \lgw N$ between $\K$-analytic manifolds $M$ and $N$. Given $\pa \in M$, we denote by $\Phi_{\pa}$ the germ of the morphism at a point $\pa \in M$, and by $\Phi_{\pa}^{\ast} : \mathcal{O}_{\Phi(\pa)} \lgw\mathcal{O}_{\pa}$ the associated  morphism of local rings, where $\mathcal{O}_{\pa}$ stands for the ring of analytic function germs at $\pa$. For each $\pa \in M$, we set $\Gr_{\pa}(\Phi) := \Gr(\Phi^{\ast}_{\pa})$ and $\Fr_{\pa}(\Phi):= \Fr(\Phi^{\ast}_{\pa})$. Consider:
\[
\RR(\Phi,M) = \{\pa \in M;\, \Gr_{\pa}(\Phi) = \Fr_{\pa}(\Phi) \},
\]
which is called the set of \emph{regular} (in the sense of Gabrielov) points of $\Phi$. We start by proving a new proof of the following result:

\begin{theorem}[Paw\l ucki Theorem I, {\cite{Paw}}]\label{thm:PawI}
Let $\Phi: M \lgm N$ be an analytic map between connected manifolds. Then $M\setminus\RR(\Phi,M)$ is a proper analytic subset of $M$.
\end{theorem}

The proof is given in $\S$\ref{sec:PawI}. The idea is to combine the uniformization Theorem (see e.g. \cite[Theorem 0.1]{BMihes}) with a new commutative algebra result, that is, the rank Theorem for W-temperate families \cite[Theorem 1.1]{BCR2} applied to Eisenstein power series, see $\S\S$\ref{ssec:Eisenstein}. Eisenstein power series have been systematically employed in the study of families of singularities, see e.g. \cite{ZarEisenstein,HirEqui,PP21}, going back at least to works of Zariski \cite[pg. 502]{ZarEisenstein}, and they play a crucial role in our paper.

\medskip

Let us now specialize our presentation to $\K =\R$, and we refer to $\S\S$\ref{sec:subanalytic} and $\S\S$\ref{sec:Nash} for all the details of the following discussion. Let $X \subset M$ be a subanalytic set. Given a point $\pa\in M$, we denote by $X_{\pa}$ the germ set of $X$ at $\pa$. We say that an equidimensional subanalytic set $X$ is a \emph{Nash set} at $\pa \in M$ (which might not belong to $X$) if there exists a germ $Y_{\pa}$ of semi-analytic set at $\pa$ such that $X_{\pa} \subset Y_{\pa}$ and $\dim(X_\pa)=\dim(Y_\pa)$. More generally, a subanalytic set $X \subset M$ of dimension $d$ is Nash at a point $\pa \in M$, if $X$ is a union of equidimensional Nash sets $\Sigma^{(k)}$, where $k=0$, \ldots, $d$. We consider the sets:
\[
\begin{aligned}
\NN(X)& :=\{\pa\in M\mid X_\pa \text{ is the germ of a Nash set}\}\\
\SA(X)&:=\{\pa\in M\mid X_\pa \text{ is the germ of a semianalytic  set}\}.
\end{aligned}
\]
It is trivially true that $M\setminus \ovl X \subset \SA(X)\subset \NN(X)$. But in general, $\SA(X) \neq  \NN(X)$, see example \ref{ex:SNneqNN} below. Now, by combining Theorem \ref{thm:PawI} with the uniformization Theorem, and following an argument from \cite{BMFourier}, we provide a new proof of the following result:

\begin{theorem}[Paw\l ucki Theorem II, {\cite{Pawthesis}}]\label{thm:PawII}
Let $X$ be a subanalytic set of a real analytic manifold $M$. Then
\begin{itemize}
\item[i)] The sets $\NN(X)$ and $\SA(X)$ are subanalytic.
\item[ii)] $\dim(M\setminus \NN(X))\leq \dim(M\setminus\SA(X))\leq \dim(X)-2$.
\end{itemize}
In particular, if $\dim(X)\leq 1$, then $\NN(X)=\SA(X)=M$.
\end{theorem}

\begin{remark}
The case of $\dim(X)\leq 1$ was originally proved by \L ojasiewicz \cite{Loj} and an alternative proof is given in \cite[Theorem 6.1]{BMihes}.
\end{remark}

The proof is given in $\S$\ref{sec:PawII}. The original proof of Theorem \ref{thm:PawII} given in \cite{Pawthesis} is an intricate construction between geometrical, algebraic, and analytic arguments, which we do not fully understand. Paw\l ucki then deduces Theorem \ref{thm:PawI} from Theorem \ref{thm:PawII} in \cite{Paw}.  Our proof of these results relies heavily on algebraic arguments, namely on \cite[Theorem 1.1]{BCR2} and the use of Eisenstein power series, see $\S\S$\ref{ssec:Eisenstein} instead of geometric and analytic arguments as in \cite{Pawthesis}. Our use of geometric techniques is essentially reduced to the extension Lemma \ref{cor_linear_equations} together with the use of the Uniformization Theorem of Hironaka \cite{HirPisa}; the former has been inspired from the work of Paw\l ucki \cite[Lemme 6.3]{Pawthesis}, while the later is not used in \cite{Pawthesis,Paw}.

\bigskip

We would like to thank Edward Bierstone for bringing the topic of this paper to our attention and for useful discussions. This work was supported by the CNRS project IEA00496  PLES. The first author is supported by the project ``Plan d’investissements France 2030", IDEX UP ANR-18-IDEX-0001. The second author thanks the grant NKFIH KKP 126683.

\section{Preliminaries in analytic and subanalytic geometry}

\subsection{Analytic set and spaces} Let $\K=\R$ or $\C$ and fix an analytic manifold $M$.

\begin{definition}[Real-analytic set]
A subset $X$ of $M$ is analytic if each point of $ M$ admits a neighborhood $U$ and an analytic function $f\in \mathcal{O}(U)$ such that:
\[
X\cap U =  \{\pa\in U; \,f(\pa)=0\}.
\]
We say that $X$ is an analytic set generated by global sections in $\mathcal{O}(M)$ if we can take $U=M$. 
\end{definition}

\begin{definition}[{cf. \cite[Ch. V, Def 6]{GR}}]
A (coherent) $\K$-analytic space is a locally ringed space $(X,\mathcal{O}_X)$, where:
\begin{itemize}
\item[(1)] $X$ is a Hausdorff topological space and $\mathcal{O}_X$ is a coherent sheaf of functions,
\item[(2)] at each point $\pa$ of $X$ there is a neighborhood $U$ such that $(U,\mathcal{O}_X|_U)$ is isomorphic to a ringed space $(Y,\mathcal{O}_Y)$ where $Y$ is an analytic subset of an open set $V \subset \K^n$ and $\O_Y$ is its sheaf of analytic functions. That is, there exist $\K$-analytic functions $(f_1,\ldots,f_d) \in \mathcal{O}(V)$ such that:
\[
Y = \{\pa \in V ;\, f_k(\pa)=0, \, k=1\ldots,d\} \quad \text{ and } \quad \mathcal{O}_Y = \mathcal{O}_V /(f_1,\ldots,f_d).
\]
\end{itemize}
A subspace of $(X,\mathcal{O}_X)$ is an analytic space $(Z,\mathcal{O}_Z)$ such that $Z \subset X$ and the inclusion $i: Z \lgw X$ is an injection that is, an injective map such that $i^{\ast}:\mathcal{O}_X \lgw\mathcal{O}_Z$ is surjective.
\end{definition}

If $\K =\R$, then it is not true that every $\R$-analytic set $X$ admits the structure of a $\R$-analytic space, as illustrated by examples of Cartan, see e.g. \cite[Ch. V, $\S$3]{n66}. In contrast, if $\K=\C$, then every $\C$-analytic set $X$ admits the structure of $\C$-analytic space, essentially by a Theorem of Oka, see e.g. \cite[Ch. VII, Th 7.1.5]{Hor}. We refer to \cite[page 155]{GR} for a definition of \emph{irreducible} complex analytic subspace $X \subset M$, and we recall that if $X$ is irreducible then it is not the union of two proper complex analytic sets $Y,Z\, \subset M$, that is, if $X = Y\cup Z$ then either $Y = X$ or $Z = X$. 

\begin{remark}\label{rk:Irreducible}
Note that if $X \subset \Omega \subset \C^n$ is an irreducible complex analytic set generated by global sections in a connected open set $\Om$, then the ring $\mathcal{O}(X)$ is an integral domain.
\end{remark}

\subsection{Semianalytic and Subanalytic sets}\label{sec:subanalytic}
We follow the presentation of \cite{BMihes}. Fix a real analytic manifold $M$.

\begin{definition}[Semianalytic set]
A subset $X$ of $M$ is semianalytic if each point of $ M$ admits a neighborhood $U$ and analytic functions $f_i \in \mathcal{O}(U)$ and $g_{i,j} \in \mathcal{O}(U)$ for $i =1$, \ldots, $p$ and $j=1$, \ldots, $q$ such that:
\[
X\cap U = \bigcup_{i=1}^p \{\pa\in U; \,f_i(\pa)=0,\, g_{i,j}(\pa)>0,\, 1,\ldots,q\}.
\]
\end{definition}

\begin{definition}[Subanalytic set]
A subset $X$ of $M$ is subanalytic if each point of $ M$ admits a neighborhood $U$ such that $X \cap U$ is the projection of a relatively compact semi-analytic set.
\end{definition}

The following is an important general example of subanalytic set:

\begin{example}\label{ex:SubanalyticProjection}
Let $\varphi : N \lgw M$ be a proper analytic map. The image $X=\varphi(N)$ is a subanalytic set of $M$. Indeed, note that the graph $\Gamma(\varphi) \subset M \times N$ is a closed analytic set and that the set $X$ is the projection of $\Gamma(\varphi)$ onto $M$, that is, the image of $\Gamma(\varphi)$ by the projection $\pi:M \times N \lgw M$. It is now enough to remark that since $\varphi$ is proper, given a relatively compact set $U \subset M$, the intersection $\pi^{-1}(U) \cap \Gamma(\varphi)$ is relatively compact.
\end{example}

\begin{definition}
A subset $X$ of $\R^n$ is \emph{finitely subanalytic} if its image under the map
$$\pi_n:\x\in\R^n\lgm \left(\frac{x_1}{\sqrt{1+\|x\|^2}},\ldots, \frac{x_n}{\sqrt{1+\|x\|^2}}\right)\in\R^n$$
is subanalytic.
\end{definition}

\begin{remark}
Because $\pi_n$ is a semialgebraic diffeomorphism, every finitely subanalytic subset of $\R^n$ is subanalytic, but the converse is not true in general: for instance
$$X=\{(t,\sin(t))\mid t\in\R\}$$
is subanalytic but not finitely subanalytic.
\end{remark}

Let $X$ be a subanalytic set. We say that $X$ is smooth (of dimension $d$) at a point $\pa \in X$ if there exists a neighborhood $U$ of $\pa$ where $X \cap U$ is an analytic sub-manifold (of dimension $d$). The dimension of $X$ is defined as the highest dimension of its smooth points, c.f. \cite[Remark 3.5]{BMihes}. Given a subanalytic (respectively, semianalytic) set $X$ and a number $k \in \mathbb{N}$, the set of all smooth points of $X$ of dimension $k$, which we denote by $X^{(k)}$, is subanalytic \cite{Tamm}, \cite[Theorem 7.2]{BMihes} (respectively, semianalytic \cite[Remark 7.3]{BMihes}). The set of pure dimension $k$ of $X$ is the set $\Sigma^{(k)} = \overline{X^{(k)}}\cap X$, which is subanalytic. If there exists $d\in \mathbb{N}$ such that $X = \Sigma^{(d)}$, we say that $X$ has pure dimension $d$. Note that $X = \cup_{k =0}^d \Sigma^{(k)}$, where $d$ is the dimension of $X$.

\begin{example}
Let $M = \R^3$ endowed with coordinate system $(x,y,z)$, and consider the Whitney umbrella $X =\{ x^2 - zy^2 =0\} \subset \R^3$. Then:
\[
\Sigma^{(2)} = \{ x^2 - zy^2 =0 \text{ and } z\geq 0\}, \quad \Sigma^{(1)} = \{x=y=0, \text{ and } z\leq 0\}.
\]
Note that their intersection is non-empty.
\end{example}

We now recall a classical result about subanalytic sets due to Hironaka \cite{HirPisa}; we follow the presentation of \cite[Theorem 0.1]{BMihes}:

\begin{theorem}[Uniformization Theorem I]\label{thm:uniformaizionSubA} Let $X \subset M$ be a closed subanalytic set of dimension $d$. There exists an analytic manifold $N$ of dimension $d$ and a proper analytic map $\varphi: N \lgw M$ such that $\varphi(N) =X$.
\end{theorem}

In what follows, we use the following variant of the above result:

\begin{theorem}[Uniformization Theorem II]\label{cor:uniformizationSubA}
Let $X \subset M$ be a closed subanalytic set of dimension $d$. There exists $d+1$ analytic manifolds $N_k$, where $k=0,\ldots,d$, where the dimension of $N_k$ is equal to $k$, and $d+1$ proper and generically immersive analytic maps $\pi_k:N_k  \lgw M$ such that $\pi_k(N_k) = \Sigma^{(k)}$. 
\end{theorem}
\begin{proof}
It is enough to prove the result when $X$ is an equidimensional subanalytic set, that is, when $X = \Sigma^{(d)}$. Let $\varphi: N \lgw M$ be the proper analytic map given by Theorem \ref{thm:uniformaizionSubA} such that $\varphi(N) =X$. We note that $N = \cup_{\iota \in I} N_{\iota}$ where each $N_{\iota}$ is a connected manifold and $I$ is an index set. Denote by $\phi_{\iota} := \varphi|_{N_{\iota}}:N_{\iota} \lgw M$. Note that the generic rank of $\varphi$ is constant along connected components of $N$, and denote by $r_{\iota}$ the generic rank associated to each $\varphi_{\iota}$. Let $J \subset I$ be the subindex set of $\iota \in I$ such that $r_{\iota} = d$; since $\varphi(N) =X$ is of dimension $d$, we conclude that $J \neq \emptyset$ and that $r_{\iota} < d$ for every $\iota \in I \setminus J$. We consider the manifold $N_d = \cup_{\iota \in J} N_{\iota}$ and the associated proper analytic morphism $\varphi_{d} : N_d \lgw M$, which we claim to satisfy all properties of the Theorem.

Indeed, we start by noting that $X \setminus \varphi_d(N_d)$ is a subanalytic set of dimension smaller than $d$ and, therefore, the closure of $\varphi_{d}(N_d) $ is equal to $X$. Since $\varphi$ is proper and continuous, we conclude that $\varphi_d(N_d) = X$. It is now enough to prove that the mapping is generically immersive. This easily follows from the fact that $\varphi$ is generically of the same rank as the dimension of $N_d$.
\end{proof}

We finish this section with a sufficient condition for a subanalytic to be analytic: 

\begin{lemma}[{\cite[Lemma 3]{Paw}}]\label{lem:Paw3}
Let $X \subset M$ be a subanalytic set which is a union of countably many analytic subsets. Then $X$ is an analytic set.
\end{lemma}
\begin{proof}
We claim that if $X$ is a subanalytic set contained in a union of countably many analytic subsets $(Y_k)_{k\in \mathbb{N}}$, then it is locally contained in a union of a finite number of the analytic sets $(Y_k)_{k\in \mathbb{N}}$. Note that the lemma easily follows from the claim. Since $X = \cup \, \Sigma^{(k)}$, where $\Sigma^{(k)}$ is a subanalytic equidimensional set, it is enough to prove the claim in the case that $X$ is an equidimensional set. By the uniformization Theorem \ref{cor:uniformizationSubA} there exists a proper analytic map $\varphi:N \lgw M$ such that $\varphi(N) =\overline{X}$ and $\varphi$ is generically of rank $d = \dim(X)$; the later condition implies that $\varphi^{-1}(X)$ is subanalytic set of $N$ whose interior is dense in $N$. Let us fix $\pa \in \overline{X}$; since $\varphi$ is proper, the fiber $\varphi^{-1}(\pa)$ has a finite number of connected components $T_1,\ldots, T_r$; denote by $U_1,\ldots,U_r$ connected open neighborhoods of the $T_k$. Now, given an analytic subset $Y\subset M$, its pre-image $Z =\varphi^{-1}(Y)$ is analytic in $N$. It follows that for each $k=1$, \ldots, $r$, either $Z\cap U_k =U_k$, or $Z \cap U_k$ is a closed set with empty interior in $U_k$. Since $X$ is contained in countable many analytic sets, and the union of countable many closed sets with empty interior has empty interior by Baire's Theorem, we conclude that for each $k=1,\ldots, r$, there is an analytic set $Y_k \subset X$ such that $\varphi^{-1}(Y_k) \cap U_k = U_k$. We conclude easily.
\end{proof}

\subsection{Regular locus of analytic maps} Let $\K=\R$ or $\C$. Consider an analytic map $\Phi: \Om\subset \K^m\lgw \K^n$ where $\Om$ is an open set. The set of regular points of $\Phi$ is given by:
\[
\RR(\Om) = \{\pa \in \Om;\, \Gr_{\pa}(\Phi) = \Fr_{\pa}(\Phi) \}.
\]
We recall that Gabrielov's rank Theorem \cite{Ga1,BCR} states that:
\[
 \Gr(\Phi_{\pa}) = \Fr(\Phi_{\pa}) \implies \Gr(\Phi_{\pa}) = \Fr(\Phi_{\pa}) = \Ar(\Phi_{\pa}).
\]
In particular, the set $\RR(\Om)$ is open. As a matter of fact it also contains a non-empty analytic-Zariski set:

\begin{lemma}\label{lem:induc0}
Let $\Phi: \Om\subset \K^m\lgw \K^n$ be an analytic map. Then the set
\[
\RR(\Phi):=\{\pa\in \Om\mid \Phi_\pa \text{ is regular }\}
\]
contains a set of the form $\Om\setminus Z$ where $Z$ is a proper analytic set of $\Om$ generated by global equations in $\mathcal{O}(\Om)$.
\end{lemma}

\begin{proof}
It is enough to prove the Lemma in the case that $\Om$ is connected. Let $r$ be the generic rank of $\Phi$ and denote by $Z$ the set of points $\pa \in\Om$ where the rank of $\Phi$ is smaller than $r$. Note that $F$ is a proper analytic subset generated by global equations in $\mathcal{O}(\Om)$; indeed, it is the zero set of the $r$-minors of the Jacobian of $\Phi$. It is now enough to note that $\Phi$ is regular at every point of $\Om \setminus Z$ by the constant rank Theorem.
\end{proof}

We now recall a result that relates the regular locus of complex and real analytic morphisms due to Milman \cite{Mil1}, but which we state as in \cite{Paw}:

\begin{lemma}[{\cite[Lemma 4]{Paw}}]\label{lem:Paw4}
Let $\Phi: \Om \subset \mathbb{C}^m \lgw\mathbb{C}^n$ be a complex analytic map and denote by $\Phi^{\mathbb{R}}$ its real-analytic counterpart. Then $\RR(\Phi,\Om) = \RR(\Phi^{\mathbb{R}},\Om)$.
\end{lemma}
\begin{proof}
The inclusion $\RR(\Phi,\Om) \subset \RR(\Phi^{\mathbb{R}},\Om)$ is immediate. In order to prove the other inclusion, suppose that $\Phi^{\mathbb{R}}$ is regular at $\pa$ and denote by $r = \Gr_{\pa}(\Phi^{\mathbb{R}})$. Since $\Phi^{\mathbb{R}}$ is the real-analytic counterpart of $\Phi$, $r = 2s$ where $s= \Gr_{\pa}(\Phi)$. The result is now immediate from \cite[Theorem 2]{Mil1}.
\end{proof}

\subsection{The Nash and the Semianalytic locus}\label{sec:Nash}

Given a subanalytic set $X \subset M$ and a point $\pa \in M$, we will denote by $X_{\pa}$ the germ set of $X$ at $\pa$.

\begin{definition}[Nash points]
Let $X \subset M$ be a subanalytic set of pure dimension $d$. We say that $X$ is a \emph{Nash set} at $\pa \in M$ (which might not belong to $X$) if there exists a germ $Y_{\pa}$ of semi-analytic set at $\pa$ such that $X_{\pa} \subset Y_{\pa}$ and $\dim(X_\pa)=\dim(Y_\pa)$. More generally, a subanalytic set $X \subset M$ of dimension $d$ is Nash at a point $\pa \in M$, if $\Sigma^{(k)}_{\pa}$ is Nash for each $k=0$, \ldots, $d$. We consider the set:
\[
\begin{aligned}
\NN(X)&:=\{\pa\in M\mid X_\pa \text{ is the germ of a Nash set}\}
\end{aligned}
\]
We say that $X$ is a Nash set if it is Nash at every point, that is, if $\NN(X) =M$. 
\end{definition}

It is clear that every semi-analytic set is Nash subanalytic.  A more general example is given by the following Lemma:

\begin{lemma}\label{ex:subAnalyticNash}
Let $\varphi : N \lgw M$ be a proper and regular analytic map, that is, at every point $\pa \in N$, $ \Gr_{\pa}(\phi)=\Fr_{\pa}(\phi) = \Ar_{\pa}(\phi)$. Suppose that $X = \phi(N)$ is equidimensional of dimension $d$. Then $X$ is Nash subanalytic.
\end{lemma}
\begin{proof}
Indeed, fix a point $\pb \in X$. Consider a relatively compact neighborhood $V$ of $\pb$, and note that $\phi^{-1}(V) =U$ is a relatively compact open set of $N$. Now, for each point $\pa \in \overline{U}$, it follows from the regularity of the mapping that there exists an open neighborhood $U_{\pa}$ of $\pa$ and a semi-analytic set $Y_{\pa} \subset M$ of dimension at most $d$ such that $\phi(U_{\pa}) \subset Y_{\pa}$. From the relative compactness of $U$, it follows that there exists a semi-analytic set $Y$ of dimension at most $d$ (given as the union of a finite number of sets $Y_{\pa}$) such that $\phi(U) \subset Y$, finishing the proof.
\end{proof}

Indeed, we may generalize the above idea to provide a description of the Nash locus in terms of the regular points of a morphism:

\begin{lemma}\label{cor:subAnalyticNashRegCondition}
Let $\varphi:N \lgw M$ be a proper generically immersive analytic morphism such that $\varphi(N) = X$ is a closed equidimensional set. Then 
\[
X \setminus \NN(X) = \varphi(N \setminus \RR(\varphi,N)).
\]
\end{lemma}
\begin{proof}
First, let us show that $X \setminus \NN(X) \subset \varphi(N \setminus \RR(\varphi,N))$ by proving the associated inclusion of their complements. Fix a point $\pb \in X\setminus \varphi(N \setminus\RR(\varphi,N))$. This means that  $\varphi$ is regular on the pre-image of $\varphi^{-1}(\pb)$. Since being regular is an open property, there exists a neighborhood $U$ of $\varphi^{-1}(\pb)$ such that $\varphi|_{U}$ is everywhere regular. Moreover, since $\varphi$ is proper and continuous, there exists a neighborhood $V$ of $\pb$ such that $\varphi^{-1}(V) \subset U$. By Lemma \ref{ex:subAnalyticNash} applied to $X \cap V$, we conclude that $\pb \in \NN(X)$ as desired.

Now, let us prove that $\varphi(N \setminus\RR(\varphi,N)) \subset X \setminus \NN(X)$ by proving the associated inclusion of their complements. Fix a point $\pb \in \NN(X)$ and let $Y_{\pb}$ be the germ of a semi-analytic set of dimension $d$ which contains $X_{\pb}$; let $V$ be a subanalytic and relatively compact neighborhood of $\pb$ where $Y_{\pb}$ admits a representative $Y$ defined in $V$ such that $X \cap V \subset Y$. Let $U =\varphi^{-1}(V)$, which is a relatively compact neighborhood of $\varphi^{-1}(\pb)$. It follows that $\varphi(U) \subset Y$, which implies that $\varphi$ is regular at every point $\pa \in U$ ; in particular, at every point $\pa \in \varphi^{-1}(\pb)$. We conclude that $\pb \notin \varphi(N \setminus\RR(\varphi,N))$, finishing the proof.
\end{proof}

We now consider the following set:
\[
\begin{aligned}
\SA(X)&:=\{\pa\in M\mid X_\pa \text{ is the germ of a semianalytic  set}\}.
\end{aligned}
\]
It is trivially true that $M\setminus \ovl X \subset \SA(X)$ and $\SA(X)\subset \NN(X)$. But in general, $\SA(X) \neq  \NN(X)$ as is illustrated by the following examples:

\begin{example}\label{ex:SNneqNN}\hfill
\begin{enumerate}
\item[i)] Consider a subanalytic two dimensional set $S$ in $\R^3$ such that the germ at the origin $S_0$ is not semianalytic (for instance, the image of a compact set through the Osgood mapping \cite{Os} provides such a surface). We consider $X:=\R^3\setminus S$; $X$ is subanalytic and of pure dimension 3, thus it is Nash subanalytic since $X\subset \R^3$. But the germ $X_0$ is not semianalytic. Note that $0\notin X$.
\item[ii)] We may modify the example as follows: we set 
\[
X:=\R^4\setminus (\R^3\times\{0\})\cup(S\times\{0\}\times\{0\}).
\]
Then $X$ is equidimensional of dimension 4, and $\NN(X)=\R^4$, but $X_0$ is not semianalytic. Note that $0\in X$.
\end{enumerate}
\end{example}

\begin{remark}\label{rk:NNforNonClosedSet}
We recall that the closure of a semianalytic (respectively, a subanalytic) set is semianalytic (respectively, subanalytic) set of the same dimension. It follows that $\NN(X) = \NN(\overline{X})$ for every subanalytic set $X \subset M$. In contrast, we can only conclude from this argument that $\SA(X) \subset \SA(\overline{X})$, c.f. example \ref{ex:SNneqNN}(i).
\end{remark}

\section{Eisenstein power series and families of morphisms}

\subsection{Eisenstein power series}\label{ssec:Eisenstein}

Given a closed polydisc $D \subset \C^n$, we denote by $\mathcal{O} = \mathcal{O}(D)$ the ring of analytic functions defined in a neighborhood of $D$, and note that it is an UFD by \cite{Da}. Let $\KK$ be an algebraic closure of its fraction field. The ring of Eisenstein series over $\O$ is the filtered limit of rings:
$$
\bigcup_{\mathfrak c\in\KK}\bigcup_{f\in \O\setminus\{0\}} \O_f\lb x_1,\ldots, x_n\rb[\mathfrak c]
$$
where $\O_f$ denotes the localization of $\O$ with respect to the multiplicative family $\{1,f,f^2,\ldots,\}$. We denote by $\KK \nl x_1,\ldots, x_r \nr$ this family of rings. Now let us consider a ring homomorphism:
\[
\phi : \KK\nl \x \nr\lgw \KK\nl \ub \nr
\]
where $\x = (x_1,\ldots,x_n)$ and $\ub = (u_1,\ldots, u_m)$. We say that $\phi$ is a morphism of Eisenstein power series if $\phi(f) = f(\phi(\x))$ for every $f\in \KK\nl \x \nr$. We denote by $\wdh{\phi}$ its extension to the ring of formal power series and we define:
\begin{equation}\label{eq:DefRanks}
\begin{aligned}
\text{the Generic rank:} & &\Gr(\phi) &:= \mbox{rank}_{\mbox{Frac}(\KK \nl \ub \nr )}(\mbox{Jac}(\phi)),\\
\text{the Formal rank:}& &\Fr(\phi) &:= \dim\left(\frac{\KK \lb \x \rb  }{\Ker(\wdh{\phi})}\right),\\
\text{and the temperate rank:}& &\Tr(\phi)&:=\dim\left(\frac{\KK \nl \x \nr}{\Ker(\phi)}\right),
\end{aligned}
\end{equation}
of $\phi$, where $\mbox{Jac}(\phi)$ stands for the matrix $[\partial_{u_i} \phi(x_j)]_{i,j}$. It follows from \cite[Th. 1.1 and Prop. 4.9]{BCR2} that:

\begin{theorem}[Eisenstein power series rank Theorem]\label{thm:TemperateRank}
Let $\phi:\KK\nl \x \nr\lgw \KK\nl  \ub \nr$ be a morphism of rings of Eisenstein power series. Then
\[
\Gr(\phi)=\Fr(\phi)\Longrightarrow \Gr(\phi)=\Fr(\phi)=\Wr(\phi).
\]
\end{theorem}

\subsection{Families of morphisms}

\begin{definition}\label{def:admissibleFamily}
Consider two analytic maps $\Phi:\Om \subset \K^m\lgw \K^n$ and $\varphi:\Lambda\subset \K^l \lgw\Om$, where $\Om$ is a connected open set and one of the following holds:
\begin{enumerate}
\item $\Lambda = \Om$ and $\varphi$ is the identity;
\item $\Lambda$ is a connected open set and $\varphi$ is an analytic map;
\item $\Lambda \subset \Om$ is an analytic subspace of $\Om$ such that $\mathcal{O}(\Lambda)$ is an integral domain, and $\varphi$ is its inclusion.
\end{enumerate}
An \emph{admissible family of analytic germs} (associated to $\Phi$ and $\varphi$) is the analytic map
\[
\begin{array}{ccc}\Psi:\Lambda\times (\K^m,0)& \lgw& (\K^n,0)\end{array}
\]
given by $\Psi(\pa,\ub) = \Phi(\varphi(\pa)+\ub) - \Phi(\varphi(\pa))$. We denote by $\Psi_{\pa}:(\K^m,0) \lgw(\K^n,0)$ the associated germ at $\pa$; in particular $\Psi_{\pa} = \Phi_{\pa}- \Phi(\varphi(\pa))$.
\end{definition}

\begin{lemma}\label{rk:BasicProperties} Given an admissible family of analytic germs:
\begin{enumerate}
\item The generic rank is constant along $\Lambda$, that is, 
\[
\forall \pa, \pb\in \Lambda,\ \Gr(\Psi_\pa)=\Gr(\Psi_{\pb}).
\]
\item The map $\pa\in\Lambda\lgm \Ar(\Psi_{\pa}^*)\in\N$ is upper semi-continuous for the Euclidean topology.

\item The ring of global sections $\mathcal{O}(\Lambda)$ is an integral domain.
\end{enumerate}
\end{lemma}
\begin{proof}
Condition $(1)$ and $(3)$ are straightforward. In order to prove $(2)$, let $f_1$, \ldots, $f_s$ be generators of $\Ker(\Phi^*_{\varphi(\pa)})$ and $U$ be an open neighborhood of $\varphi(\pa)$ such that the $f_i$ are well defined on $U$. Let $V$ be a connected neighborhood of $\pa$ contained in $\varphi^{-1}(U)$. Since $\Phi$ is analytic, apart from shrinking $U$ and $V$, we have that $f_i \circ \Phi \circ\varphi_{\pb} \equiv 0$, for all $\pb \in V$. We conclude easily.
\end{proof}

Now fix an admissible family of analytic map germs 
\[
\Psi:\Lambda\times (\K^m,0) \lgw (\K^n,0).
\]
and let $\LLL$ denote the fractions field of the ring $\O(\Lambda)$ of analytic functions on $\Lambda$. Note that $\Psi$ induces a morphism of power series rings:
\[
\Psi_\LLL^*: \LLL\lb \x\rb   \lgw \LLL\lb \ub\rb
\]
where $\x=(x_1,\ldots, x_n)$, $\ub=(u_1,\ldots, u_m)$ and
\[
\Psi_\LLL^*(x_i) = \sum_{\gamma \in \mathbb{N}^m \setminus 0} F_{i,\gamma} \ub^{\gamma}, \quad F_{i,\gamma} =\frac{1}{\g!} \frac{\partial^{|\gamma|}}{\partial \w^{\gamma}} (x_i \circ \Phi) \circ \varphi \in \mathcal{O}(\Lambda).
\]
where $\w = (w_1,\ldots,w_m)$ are globally defined coordinate systems over $\Omega$. Note that $F_{i,0} =0$ for every $i =1,\ldots,n$, which guarantees that $\Psi_\LLL^*$ is well-defined.

Now let $r=\Gr(\Psi_{\LLL}^{\ast})$. Thus any $(r+1)\times (r+1)$ minor of the Jacobian matrix of $\Phi_\LLL^{\ast}$ is zero, therefore $\Gr(\Psi_\pa)\leq r$ for every $\pa\in \Lambda$. On the other hand, there is a $r\times r$ minor of the Jacobian matrix of $\Phi_\LLL^\ast$, denoted by $M$, that is not identicaly zero. So, for a generic $\pa\in\Lambda$, we have $M(\pa)\neq 0$ and $\Gr(\Psi_\pa)=r$. Therefore, 
by Lemma \ref{rk:BasicProperties}(1), we have that:
\[
\Gr(\Psi_{\LLL}^{\ast}) = \Gr(\Psi_{\pa}), \quad \forall \pa \in \Lambda.
\]
 We now turn to the problem of relating the formal rank of $\Psi$ at a point $\pa \in \Lambda$ with the formal rank of $\Psi_\LLL^*$:

\begin{proposition}\label{cor_extension}
Let $\Psi:\Lambda\times (\K^m,0) \lgw (\K^n,0)$ be an admissible family of analytic map germs. If there is $\pa \in \Lambda$ such that $\Gr(\Psi_{\pa}) = \Fr(\Psi_{\pa})$, then:
\begin{equation}\label{eq:ranks}
\Gr(\Psi_{\LLL}^*)=\Fr(\Psi_{\LLL}^*).
\end{equation}
In particular, the set
\[
\RR(\Psi,\Lambda):=\{\pa\in \Lambda \mid \Psi_\pa \text{ is regular }\}
\]
is either empty or contains a set of the form $\Lambda \setminus W$ where $W$ is a countable union of proper analytic subsets of $\Lambda$ generated by global equations in $\mathcal{O}(\Lambda)$.
\end{proposition}

The proof of this Proposition is based on an extension result, namely Lemma \ref{cor_linear_equations} below, whose proof is strongly inspired by an argument of Paw\l ucki cf. \cite[Lemme 6.3]{Pawthesis}. We postpone the proof to $\S$\ref{sec:Extention}. Condition \eqref{eq:ranks} is the deepest statement of the above Proposition which, together with Theorem \ref{thm:TemperateRank}, allows us to prove the following crucial technical result:

\begin{theorem}\label{lem:AlgebraGeometryHook}
Let $\Psi : \Lambda \times (\K^m,0) \lgw(\K^n,0)$ be an admissible family of analytic germs where $\Lambda$ is a connected open set of $\K^l$ (that is, we consider cases (1) and (2) of Definition \ref{def:admissibleFamily}). Then either $\RR(\Psi,\Lambda) = \emptyset$ or, for every $\pa \in \Lambda$, there exists an open neighborhood $U_{\pa} \subset \Lambda_{\pa}$ and a proper analytic set $Z \subset U_{\pa}$ such that $\RR(\Psi,U_{\pa}) \supset U_{\pa }\setminus Z$.
\end{theorem}
\begin{proof}
Let $\LLL$ denote the fraction field of $\mathcal{O}(\Lambda)$. Note that $\Psi$ yields a morphism $\Psi_{\LLL}^{\ast}:\LLL \lb \x \rb \lgw\LLL \lb \ub \rb$ and that $\Gr(\Psi_{\LLL}^{\ast}) = \Gr(\Psi_{\pa})$ for any $\pa \in \Lambda$. Now, suppose that; $\RR(\Psi,\Lambda) \neq \emptyset$ so that Proposition \ref{cor_extension} yields:
\[
\Gr(\Psi_\LLL^*) = \Fr(\Psi_{\LLL}^{\ast}).
\]
We now first prove the Lemma in the case that $\K=\C$. Let $\pa \in \Lambda$ be fixed and consider a sufficiently small closed polydisc $D_{\pa} \subset \Lambda$ centered at $\pa$. Let $ \mathcal{O}(D_{\pa})$ denote the ring of analytic functions defined in a neighborhood of $D_{\pa}$; note that this ring is a UFD by \cite{Da}. Let $\KK$ denote the algebraic closure of the fraction field of $ \mathcal{O}(D_{\pa})$. We note that the restriction of $\Psi$ to $D_{\pa}$, yields a temperate morphism $\Psi^{\ast}_{\KK}: \KK \nl \x \nr \lgw\KK \nl \ub \nr$. It is clear that $\Gr(\Psi^{\ast}_{\KK}) = \Gr(\Psi^{\ast}_{\LLL})$, and since the restriction from $\Lambda$ to $D_{\pa}$ yields an injective morphism from $\mathcal{O}(\Lambda) $ into $ \mathcal{O}(D_{\pa})$, we conclude that:
\[
 \Gr(\Psi_\KK^*)= \Fr(\Psi^{\ast}_{\KK}).
\]
so that we may apply Theorem \ref{thm:TemperateRank} in order to get
\[
\Gr(\Psi_\KK^*) = \Wr(\Psi_\KK^*)=:r.
\]
Now, up to a $\KK$-linear change of coordinates, applying \cite[Prop. 2.8 vi)]{BCR2}, the morphism $\KK\nl x_1,\ldots,x_r\nr \lgw\fract{\KK\nl \x \nr}{\Ker(\Psi_\KK^*)}$ is finite, which means that there are non-zero Weierstrass polynomials 
\[
Q_i(x_1,\ldots, x_r,x_{r+i}) \in \KK \nl x_1,\ldots, x_r \nr[x_{r+i}]\text{ for } i=1,\ldots, n-r,
\]
such that $\Psi^{\ast}_{\KK} (Q_i ) \equiv 0$. By the definition of $\KK\nl\x\nr$ and the primitive element theorem, there exists $f\in \mathcal{O}(D_{\pa}) $ and $\pc \in \KK$ of degree $d$ such that $Q_i \in \mathcal{O}(D_\pa)_f \lb \x \rb [\pc]$, that is
\[
Q_i=\sum\limits_{j=0}^{d-1} Q_{i,j} \pc^j, \quad Q_{i,j} \in \mathcal{O}(D_\pa)_f \lb\x \rb .
\]
Note that $\Psi_\KK^*(\pc)=\pc$ and $\{1, \pc, \ldots, \pc^{d-1}\}$ are linearly independent over $\mathcal{O}(D_{\pa})$. Hence, up to replacing $Q_i$ by $Q_{i,0}$, which is monic, we can choose the $Q_i$ in $\mathcal{O}(D_\pa)_f \lb \x \rb$. Let $U_{\pa} \subset D_{\pa}$ be any open neighborhood of $\pa$. We set $Z = \{ \pb \in U_{\pa}; \, f(\pb)=0\}$. Note that $Q_i$ yields a power series $Q_{i,\pb} \in \C \lb \x \rb$ at each $\pb \in U_{\pa} \setminus Z$ and that $\Psi_{\pb}^{\ast}(Q_{i,\pb}) \equiv 0$, for every $i=1,\ldots, n-r$. We conclude that $\Gr(\Psi_{\pb}^{\ast}) = \Fr(\Psi_{\pb}^{\ast})$ for every $\pb \in U_{\pa} \setminus Z$ as we wanted to prove.

Now let us consider the case that $\K = \R$. Denote by $\Lambda^{\C}$ a complex open neighborhood of $\Lambda$ such that $\Lambda^{\C} \cap \R^l = \Lambda$, over which $\Psi$ admits an holomorphic extension:
\[
\Psi^{\C} : \Lambda^{\C} \times (\C^m,0) \lgw(\C^n,0).
\]
By the first part of the proof, for each $\pa \in \Lambda^{\C}$, there exists a neighborhood $U_{\pa}^{\C}$ and a complex analytic set $Z^{\C} \subset U_{\pa}^{\C}$ such that $\RR(\Psi,U_{\pa}) \supset U_{\pa }\setminus Z$. We fix a point $\pa \in \Lambda$ and we consider the neighborhood $U_{\pa} = U_{\pa}^{\C} \cap \R^l$ and the intersection $Z := Z^{\C} \cap U_{\pa}$. It is now enough to note that $Z$ is a proper real-analytic subset of $U_{\pa}$.
\end{proof}

\section{Proof of Theorem \ref{thm:PawI}} \label{sec:PawI}

We start by a well-known result, which follows from the geometrical statement of Proposition \ref{cor_extension}:

\begin{proposition}[cf. {\cite[Prop. 1]{Paw}}]\label{prop:count_ana}
Let $\Phi:\Om\subset \K^m\lgw \K^n$ be an analytic map where $\Om$ is open. Then $\Om\setminus\RR(\Phi,\Om)$ is a union of countably many analytic subsets.
\end{proposition}

\begin{proof}
Let us first argue the case that $\K = \C$, in which case every complex analytic set is a complex analytic space. By Proposition \ref{cor_extension} applied to each connected component of $\Om$, $ X:=\Om\setminus\RR(\Phi,\Om)$ is included in the union of countably many analytic subsets $\bigcup_{i=0}^\infty Y_i$ of $\Om$. We may assume that the $Y_i$ are irreducible (in $\Om$) by replacing each $Y_i$ by its irreducible components, and we change the family $\{Y_i\}_i$ according to the following rule:
\begin{itemize}
\item[(R)] For a given $i_0$, if there is countably many irreducible analytic subspaces $Y_{i_0,k}$ of $\Om$ of dimension $<\dim(Y_{i_0})$ such that 
$X\cap Y_{i_0} \subset 
\bigcup_{k=0}^\infty Y_{i_0,k}$, we replace the family $\{Y_i\}_{i\in\N}$ by $\{Y_i\}_{i\neq i_0}\cup\{Y_{i_0,k}\}_{k\in\N}$.
\end{itemize}
By repeating this rule countably many times, we can assume that the family $\{Y_i\}_{i\in \mathbb{N}}$ is minimal in respect to $(R)$ and contains $X$. 
Now assume by contradiction that $X\neq \bigcup_{i\in\N} Y_i$. This means that there is $i_0 \in \N$ such that $Y_{i_0}\not\subset X$ but $Y_{i_0}\cap X\neq\emptyset$. By Proposition \ref{cor_extension} applied to $Y_{i_0}$ (cf. Remark \ref{rk:Irreducible} and Definition \ref{def:admissibleFamily}(3)) we have that $Y_{i_0}\cap X$ is included in a countable number of proper analytic subsets $\{Y_{i_0,k}\}_{k\in \mathbb{N}}$ of $Y_{i_0}$ that are of dimension $<\dim(Y_{i_0})$. Since $Y_{i_0}$ is an analytic subspace of $\Omega$, we conclude that each $Y_{i_0,k}$ is analytic subspace of $\Omega$, which contradicts the minimality of the family $\{Y_i\}_{i\in \mathbb{N}}$ in respect to $(R)$.

If $\K=\R$, the result follows from considering a complexification of $\Phi$, and noting that the set of regular points is non-empty by Lemma \ref{lem:induc0}.
\end{proof}

We are now ready to prove the following local version of Theorem \ref{thm:PawI}, which immediately implies it:

\begin{theorem}[Paw\l ucki Theorem I {\cite{Paw}}]\label{thm:main}
Let $\Phi:\Om\subset \K^m\lgm \K^n$ be an analytic map where $\Om$ is open. Then $\Om\setminus\RR(\Phi,\Om)$ is a proper analytic subset of $\Om$.
\end{theorem}
\begin{proof}
By Lemma \ref{lem:Paw4} and Corollary \ref{prop:count_ana}, it is enough to consider the case where $\K =\R$. Furthermore, from Lemma \ref{lem:Paw3} and Proposition \ref{prop:count_ana}, it is enough to show that $\RR(\Phi,\Om)$ is a subanalytic set of $\Om$. Note that being subanalytic is a local property, so we may suppose that $\Om$ is a subanalytic open set.

We claim that for every closed subanalytic set $X\subset \Om$, the intersection $X \cap \RR(\Phi,\Om)$ is subanalytic in $\Om$. The result then follows from the Claim applied to $X = \Om$. We prove the claim by induction on the dimension of $X$.

When $\dim(X)=0$, the result immediate. Assume the Claim is proved for $d-1\geq 0$ and let $X$ be a subanalytic subset of $\Om$ of dimension $d$. Consider its equidimensional part $\Sigma^{(d)}$ and let $E = \overline{X \setminus\Sigma^{(d)}}$, which is a closed subanalytic set of dimension $< d$. By induction $E \cap \RR(\Phi,\Om)$ is a subanalytic subset of $\Om$. It is, therefore, enough to prove the claim when $X = \Sigma^{(d)}$ is an equidimensional set.

By Corollary \ref{cor:uniformizationSubA}, there exists a proper and generically immersive analytic morphism $\varphi: N \lgw X$ such that $\varphi(N) =X$. Now fix a point $\pa \in \Lambda$ and a connected open neighborhood $\Lambda_{\pa}$ of $\pa$. We consider the family of admissible morphism:
\[ 
\begin{array}{cccc}\Psi: & \Lambda_{\pa}\times (\R^m,0)& \lgw& (\R^n,0)\\ &(\pa',\ub) & \lgm & \Phi(\varphi(\pa')+\ub) - \Phi(\varphi(\pa'))
\end{array}
\]
By Theorem \ref{lem:AlgebraGeometryHook}, apart from shrinking $\Lambda_{\pa}$, we conclude that either  $\varphi(\Lambda_{\pa}) \subset \Om \setminus \RR(\Phi,\Om)$ or there exists a analytic proper set $Z_{\pa} \subset \Lambda_{\pa}$ such that $\varphi(\Lambda_{\pa}\setminus Z_{\pa}) \subset \RR(\Phi,\Om)$. Note that, since $\pa \in N$ was arbitrary and both of these properties are open, they hold globally over each different connected component of $N$. We conclude that there exist two closed subanalytic subsets $Y$ and $Z$ of $X$, such that: $Y$ is of dimension $d$ and $Y\subset \Om \setminus \RR(\Phi,\Om)$; and $Z$ is of dimension $<d$ and $X \setminus (Y \cup Z) \subset \RR(\Phi,\Om)$. The result now follows from induction applied over $Z$. 
\end{proof}

\section{Proof of Theorem \ref{sec:PawII}}  \label{sec:PawII}

We start by proving the following Corollary of the uniformization Theorem  \ref{cor:uniformizationSubA} and Theorem \ref{thm:PawI}.

\begin{proposition}\label{prop:mainPaw}
Let $X$ be a subanalytic set of a real analytic manifold $M$. Then
\begin{itemize}
\item[i)] The set $\NN(X)$ is subanalytic.
\item[ii)] $\dim(M\setminus \NN(X)) \leq \dim(X)-2$.
\end{itemize}
In particular, if $\dim(X)\leq 1$, then $\NN(X)=M$.
\end{proposition}
\begin{proof}
By remark \ref{rk:NNforNonClosedSet}, we may suppose without loss of generality that $X$ is a closed subanalytic set. First consider the equidimensional case $X= \Sigma^{(d)}$. Denote by $\varphi:N \lgw M$ the proper generically immersive analytic morphism given by Corollary \ref{cor:uniformizationSubA}, where $N$ is of dimension $d$ and $\varphi(N) = X$. In particular $\Gr(\phi)=d$. By Theorem \ref{thm:main}, $N \setminus \RR(\varphi,N)$ is a proper analytic subset of $N$. It follows from Lemma \ref{cor:subAnalyticNashRegCondition} that $X \setminus \NN(X)$ is a subanalytic set of codimension at least $1$. It remains to prove that it has codimension $2$.

Denote by $F$ the set of points in $N$ where $\varphi$ does not have maximal rank. Note that $F$ is analytic (it is given by the zero locus of the Jacobean ideal of $\varphi$) so, apart from applying resolution of singularities, we may suppose that $F$ is a simple normal crossing divisor in $N$. Now, note that $N \setminus F \subset \RR(\varphi,N)$ since $\varphi|_{N\setminus F}$ is a local submersion. It follows that $N \setminus \RR(\varphi,N) \subset F$. So, it is enough to prove that the image $\varphi(E \setminus \RR(\varphi,N))$ has dimension at most $d-2$ for every irreducible (in particular connected) component $E \subset F$. Fix such an $E$ and consider the morphism $\varphi_E = \varphi|_E: E \lgw M$. Let $r$ denote the generic rank of $\varphi_E$ and note that $r \leq d-1$ since $E$ has dimension $d-1$. If $r<d-1$, then $\varphi(E)$ is a subanalytic set of dimension at most $d-2$ and the result is clear. So we may suppose that $r= d-1$.

Fix a point $\pa \in E$ and consider a local coordinate system $(u,v) =(u,v_1,\ldots,v_{d-1})$ of $N$ centered at $\pa$ and defined in an open neighborhood $U$ of $\pa$, such that $E \cap U = (u=0)$. From the rank condition over $\varphi_E$, and the inverse function Theorem, there exists a coordinate system $(x,y,z) = (x_1,\ldots,x_{d-1},y,z_{d+1},\ldots, z_n)$ centered at $\varphi(\pa)= \pb$ such that:
\[
\varphi^{\ast}(x_i)= v_i, \quad i=1,\ldots,d-1.
\]
Now, apart from an analytic change of coordinates in the target and a permutation of $y$ and the $z_k$, we may further suppose that there exists a positive integer $a$ such that:
\[
\begin{aligned}
\varphi^{\ast}(y)&= u^{a}g_d(u,v) &&\\ 
\varphi^{\ast}(z_k) &= u^{a}g_k(u,v),& \quad &k=d+1,\ldots,n
\end{aligned}
\]
where $g_d(0,v)\not\equiv 0$. In particular, the set of points of $E\cap U$ where $g_d(0,v) \neq 0$ is an open dense set $E'$ of $E\cap U$. We claim that at every point of $E'$, $\varphi$ is a regular mapping; this claim implies that $\RR(\varphi,N)\cap E\cap U$ is a proper analytic set of $E$ and, therefore, $\varphi(E \setminus \RR(\varphi,N))$ has dimension at most $d-2$. We turn to the proof of the Claim: suppose that $\pa$ is a point in $E'$. Apart from shrinking $U$ and making a change of coordinates in the source and target, we may further suppose that:
\[
\varphi^{\ast}(y) = u^a,
\]
and we consider the following functions defined in the target:
\[
P_k(x,y,z) = \prod_{i=1}^a ( z_{k} - y g_{k}(x, \xi^{i} y^{1/a})), \quad k=d+1, \ldots, n
\]
where $\xi$ is a primitive $a$-root of unity. By construction, it is clear that $P_k \circ \varphi|_{U} \equiv 0$ for every $k=d+1,\ldots,d$. We conclude that $ \Gr_{\pa}(\varphi)=\Ar_{\pa}(\varphi)=d$ proving the claim and finishing the proof of the Theorem in the case of an equidimensional subanalytic set $X$.

We now consider a general closed subanalytic set $X$. Consider the morphisms from Corollary \ref{cor:uniformizationSubA} $\varphi_k:N_k \lgw M$, for $k=0,\ldots,d-1$. From the previous argument applied to each set $\Sigma^{(k)}$, we conclude that $M \setminus \NN(\Sigma^{(k)})$ is a subanalytic set of dimension at most $k-2$. It follows from the definition of $\NN(X)$ that:
\[
\NN(X) = \cap_{k=0}^{d} \NN(\Sigma^{(k)})
\]
which is a subanalytic set. Furthermore, its complement is equal to the union of the complements of $\NN(\Sigma^{(k)})$, and therefore is a subanalytic set of dimension at most $d-2$, finishing the proof.
\end{proof}

We are now ready to complete the proof of Theorem \ref{thm:PawII}, following an argument from \cite{BMFourier}. We start with two lemmas (see also \cite{FG2} for the study of the relations between $\SA(X)$ and $\NN(X)$ in general):

\begin{lemma}\label{lem:decomp_SA}
Let $X$ be a subanalytic set of dimension $d$. Then
\[
\SA(X)=\SA(X\setminus X^{(d)})\cap\SA(X^{(d)}).
\]
\end{lemma}
\begin{proof}
Note that $\SA(X\setminus X^{(d)})\cap\SA(X^{(d)}) \subset \SA(X)$ is trivial. In order to prove the other inclusion, let $\pa \in \SA(X)$; in particular $X_{\pa}$ is a semi-analytic germ. Let $U$ be a sufficiently small neighborhood of $\pa$ where $X_{\pa}$ is realizable by $X \cap U$, which is semi-analytic. We recall that if $Y$ is semi-analytic, the n $Y^{(d)}$ is a semi-analytic set, see e.g. \cite[Remark 7.3]{BMihes}, so we conclude that $X^{(d)} \cap U$ is semi-analytic and $\pa \in \SA(X^{(d)})$. Since $(X\setminus X^{(d)}) \cap U = X \cap U \setminus (X^{(d)}\cap U)$, we conclude easily.
\end{proof}

\begin{lemma}[c.f. {\cite[p. 200]{BMFourier}}]\label{lem:decomp_SA2}
Let $X$ be a closed subanalytic set of equidimension $d$. Then:
\[
\SA(X^{(d)})=\SA(X \setminus X^{(d)})\cap\NN(X^{(d)}).
\]
\end{lemma}
\begin{proof}
Clearly we have $\SA(X^{(d)})\subset \NN(X^{(d)})$. Moreover,  if $\pa\in \SA(X^{(d)})$, then $X^{(d)}_\pa$ is semianalytic, so its closure, which is $X_{\pa}$, is semianalytic and $X_{\pa} \setminus X^{(d)}_{\pa}$ is semianalytic. Thus $\SA(X^{(d)})\subset \SA(Y)\cap\NN(X^{(d)})$.

In order to prove the other inclusion, let $\pa\in\SA(Y)\cap\NN(X^{(d)})$ where $Y=X \setminus X^{(d)}$. Since the result is local, apart from replacing $M$ by a sufficiently small neighborhood of $\pa$, we may suppose that $Y$ is semianalytic and that there exists  a closed analytic set $Z $ of dimension $d$ such that $X^{(d)} \subset Z$; we conclude that $X \subset Z$. Let $\mbox{Sing}(Z)$ denote the singular points of $Z$. It follows that $X \setminus (Y \cup \mbox{Sing}(Z))$ is open and closed in $Z \setminus (Y \cup \mbox{Sing}(Z))$ and, thus, $X \setminus (Y \cup \mbox{Sing}(Z))$ is semi-analytic. Since the closure of this set is equal to $X$, we conclude that $X$ is semianalytic, and we conclude by Lemma \ref{lem:decomp_SA}. 
 \end{proof}

\begin{proof}[Proof of Theorem \ref{thm:PawII}]
Because of Proposition \ref{prop:mainPaw}, it only remains to show that $ \SA(X)$ is a subanalytic set whose complement is of dimension at most $d-2$. We prove this result by induction on the dimension of $X$; the case that $d=0$ being trivial. So, fix a subanalytic set $X$ of dimension $d$ and consider the set $Y=X \setminus X^{(d)}$, which is a subanalytic set of dimension at most $d-1$. By Lemmas \ref{lem:decomp_SA} and \ref{lem:decomp_SA2} we get:
\[
\SA(X) = \SA(Y) \cap \SA(X^{(d)}) = \SA(Y) \cap \SA(\overline{X^{(d)}}\setminus X^{(d)})\cap\NN(X^{(d)}). 
\]
By induction applied to $Y$ and $\overline{X^{(d)}}\setminus X^{(d)}$, and by Proposition \ref{prop:mainPaw} applied to $X^{(d)}$, we conclude that $\SA(X)$ is a subanalytic set whose complement has dimension smaller or equal to $d-2$.
\end{proof}

We finish this section by proving the following corollary:

\begin{corollary} Let $X\subset \R^n$ be a finitely subanalytic set. Then $\NN(X)$ and $\SA(X)$ are finitely subanalytic.
\end{corollary}
\begin{proof} Let us denote by $\pi$ the map
$$
\x\in\R^n\lgm \left(\frac{x_1}{\sqrt{1+\|x\|^2}},\ldots, \frac{x_n}{\sqrt{1+\|x\|^2}}\right)\in\R^n.
$$
By hypothesis the image $Y = \pi(X)$ is a subanalytic set. By Theorem \ref{thm:PawII} $\NN(Y)$ is subanalytic. Furthermore, since $\pi$ is a semialgebraic diffeomorphism, we conclude that $\pi(\NN(X)) = \NN(Y) \cap \pi(\mathbb{R}^n)$ is subanalytic, which proves that $\NN(X)$ is finitely subanalytic.

The proof that $\SA(X)$ is finitely subanalytic is identical.
\end{proof}

\section{Proof of Proposition \ref{cor_extension}}\label{sec:Extention}

\subsection{Extension Lemma} The goal of this subsection is to prove the following:

\begin{lemma}[Extension Lemma]\label{cor_linear_equations}
Let $\Psi:\Lambda\times(\K^m,0)\lgw (\K^n,0)$ be an admissible family of analytic map germs (see Definition \ref{def:admissibleFamily}) and let $\LLL$ be the field of fractions of $\O(\Lambda)$. Let $(\x,y)$ be a coordinate system of $(\K^n,0)$ where $y$ is a distinguished variable. Let $U$ be an open and connected subset of $\Lambda$ and suppose that there exists a polynomial in $y$
\[
f(\x,y)=y^d+a_1(\pa,\x)y^{d-1}+\cdots+a_d(\pa,\x)
\]
such that
\begin{enumerate}
\item[i)] $a_i(\pa,\x)\in\O(U)\lb \x\rb$, $i=1,\ldots,d$;
\item[ii)] $a_i(\cdot,0)\equiv 0$ on $U$, $i=1,\ldots,d$;
\item[iii)] for all $\pa\in U$, $f(\pa,\x,y)$ is a generator of $\Ker(\wdh{\Psi}_\pa^*)$.
\end{enumerate}
Let us write $a_i(\pa,\x)=\sum_{\b\in\N^{n-1}} a_{i,\b}(\pa)\x^\b$.
Then, for every $i$ and $\b$, there is a proper global analytic subset $Z_{i,\b}\subsetneq \Lambda$ such that $a_{i,\b}$ extends on $\Lambda\setminus Z_{i,\b}$ as an analytic function $\ovl a_{i,\b}\in\LLL$. Moreover if we set
$$\ovl f:=y^d+\sum_{\b\in\N^{n-1}} \ovl a_{1,\b}(\pa)\x^\b y^{d-1}+\cdots+\sum_{\b\in\N^{n-1}} \ovl a_{d,\b}(\pa)\x^\b\in\LLL\lb \x\rb[y]$$
then $f(\x,y)\in\Ker(\Psi_\LLL^*)$.
\end{lemma}

The proof of this result is strongly inspired by the proof of \cite[Lemme 6.3]{Pawthesis}, and is based on Chevalley's Lemma:

\begin{proposition}[Chevalley's Lemma]\label{chevalley}\cite[Lemma 7]{Ch}
Let $\k$ be a field.
Let $\phi:\k\lb \x\rb\lgw \k\lb \ub\rb$ be a morphism of formal power series rings. Then there exists a function $\la:\N\lgw \N$ such that
$$\forall k\in\N,\ \phi^{-1}((\ub)^{\la(k)})\subset (\x)^k+\Ker(\phi).$$
The smallest function satisfying this property is called the \emph{Chevalley's function} of $\phi$, and is denoted by $\la_\phi$.
\end{proposition}

We start by fixing notation and by proving a Corollary of Chevalley's Lemma. Let  $\k$ be a field and $\phi:\k\lb \x\rb\lgw \k\lb \ub\rb$ be a morphism of formal power series rings. We set $\x':=(x_1,\ldots, x_{n-1})$. Let us consider the images of the $x_i$ by $\phi$: 
\[
\phi_i=\sum_{\a\in \N^m} \phi_{i,\a}\ub^\a
\]
where the $\phi_{i,\a}\in\k$. Let 
\begin{equation}\label{eq:universalpolynomial}
F(x):=x_n^d +A_1(\x')x_n^{d-1}+\cdots+A_d(\x')
\end{equation}
where the $A_i$ are universal power series
\begin{equation}\label{eq:universalpowerseries}
A_i:=\sum_{\b\in\N^{n-1}}A_{i,\b}\x'^\b
\end{equation}
and the $A_{i,\b}$ are new indeterminates. Then we can expand
$$
F(\phi_1,\ldots, \phi_m)=\sum_{\g\in\N^m}F_{\g}\ub^\g
$$
where 
$$
F_\g=\sum_{i,\b}  M_{\g,i,\b}A_{i,\b}+B_\g
$$
with $M_{\g,i,\b}$ and $B_\g$ polynomials in the $\phi_{j,\a}$.

Let $R$ be a ring. Then the system of linear equations
\begin{equation}\tag{$S_\infty$}\label{eq_infinity}\forall \g\in\N^m,\ \ \  F_\g(A_{i,\b})=0\end{equation} has a  solution $(a_{i,\b})\in R^\N$ if and only if $\Ker(\phi)$ contains a non zero Weierstrass polynomial
\begin{equation}\label{eqW}
f=x_n^d+a_1(\x')x_n^{d-1}+\cdots+a_d(\x'),
 \text{ where }a_i(\x')=\sum_{\b\in\N^{n-1}}a_{i,\b}{\x'}^{\b}.
 \end{equation}
  Let us consider the systems of linear equations
\begin{equation}\tag{$S_k$}\label{eq_S_k}\forall \g\in\N^m, |\g|< k,\ \ \  F_\g(A_{i,\b})=0\end{equation} 
where $k$ runs over $\N$. We have

\begin{corollary}[Approximation]\label{app1}
Let $\k$ be a field. Assume that $f$, given as in \eqref{eqW}, is a generator of $\Ker(\phi)$. Then  $(a_{i,\b})$ is the unique solution of \eqref{eq_infinity} in $\k^\N$. Moreover, there is a function $\mu:\N\lgw \N$ such that, for all
$k\in\N$, all solutions $(\wdt a_{i,\b})\in\k^{\N}$ of $(S_{\mu(k)})$ satisfies
$$
\forall \b\in\N^n, \ |\b|\leq k \Longrightarrow \wdt a_{i,\b}=a_{i,\b}.
$$
\end{corollary}

\begin{proof}
Let $(\wdt a_{i,\b})$ be a solution of \eqref{eq_infinity}. Then $$\wdt f :=x_n^d+\sum_{\b\in\N^{n-1}}\wdt a_{1,\b}\x^\b x_n^{d-1}+\cdots+\sum_{\b\in\N^{n-1}}\wdt a_{d,\b}\x^\b\in \Ker(\phi).$$
Since $f$ is a generator of $\Ker(\phi)$, there is $g\in\k\lb \x\rb$ such that $\wdt f=fg$. Since $f$ and $\wdt f$ are Weierstrass polynomials, by the uniqueness of the decomposition of a series as a product of a Weierstrass polynomials with a unit, we have that $g=1$ and $\wdt f=f$. This shows that $(a_{i,\b})$ is the unique solution of \eqref{eq_infinity}. Next, for $k\in\N$ we set 
\[
\mu(k)=\la\left((d+1)^d(k+1)  \right)
\]
where $\la$ is given in Proposition \ref{chevalley}. Consider a solution $(\wdt a_{i,\b})\in\k^{\N}$ of $(S_{\mu(k)})$. Set
\[
\wdt f:=x_n^d+\wdt a_1(\x')x_n^{d-1}+\cdots+\wdt a_d(\x'), \quad \text{ where }\quad \wdt a_i:=\sum_{\b\in\N^{n-1}}\wdt a_{i,\b}\x^\b, \, i=1,\ldots,d.
\]
Since $\phi(\wdt f)\in (\ub)^{\mu(k)}$, by Proposition \ref{chevalley},
$\wdt f\in (x)^{(d+1)^d(k+d+1)}+\Ker(\phi)$. Therefore
$$
\wdt f=fg+\sum_{i=1}^d(\wdt a_i-a_i)x_n^{d-i}$$ for some $g$, where  $$ \sum_{i=1}^d(\wdt a_i-a_i)x_n^{d-i}\in \Ker(\phi)+(\x)^{(d+1)^d(k+d+1)}.$$
Thus we can write
\begin{equation}\label{rel_ini}
\sum_{i=1}^d(\wdt a_i-a_i)x_n^{d-i}=fh+\e
\end{equation} 
where $\e\in(\x)^{(d+1)^d(k+d+1)}$. 
We denote by $\nu$ the monomial valuation defined by
$$\nu\left(\sum_{\a\in\N^n}g_\a x^\a\right):=\min\{(d+1)(\a_1+\cdots+\a_{n-1})+\a_n\mid g_\a\neq 0\}.$$
For a power series $g$, we denote by $\ini(g)$ its initial term in respect to this monomial valuation. We remark that, for any $g$, $(d+1)\ord(g)\geq \nu(g)\geq \ord(g)$.\\
Note that $\ini(f)=x_n^d$. But, in \eqref{rel_ini}, we see that the initial term of the left hand side is not divisible by $x_n^d$.  Therefore $\nu\left(\sum_{i=1}^d(\wdt a_i-a_i)x_n^{d-i}\right)\geq\nu(\e)$. Therefore
$$(d+1)\ord\left(\sum_{i=1}^d(\wdt a_i-a_i)x_n^{d-i}\right)\geq \ord(\e).$$ Thus, there is a $i_0$ such that $$\ord((\wdt a_{i_0}-a_{i_0})x_n^{d-i_0})\geq (d+1)^{d-1}(k+d+1).$$
 In particular $  \wdt a_{i_0}-a_{i_0}\in (\x)^{(d+1)^{d-1}(k+d+1)-(d-i_0)}\subset(\x)^{k+1}$.
On the other hand we have  that $\sum_{i\neq i_0}(\wdt a_i-a_i)x_n^{d-i}\in \Ker(\phi)+(\x)^{(d+1)^{d-1}(k+d+1)}$. The result is proved by induction  on the number of terms in the sum.
\end{proof}
We are now ready to turn to the proof of the main result of this subsection:

\begin{proof}[Proof of the Extension Lemma \ref{cor_linear_equations}]
We consider, for each $\pa\in U$, the following system of linear equations
\begin{equation}
\tag{$S_\infty(\pa)$}\label{eq_infinity_par}
\forall \g\in\N^m,\ \ \  F_{\g}(\pa)(A_{i,\b})=0
\end{equation} 
where $F,A_i$ are as in equations \eqref{eq:universalpolynomial} and \eqref{eq:universalpowerseries}, respectively. Set $\Psi_k =  \pi_k\circ \Psi$ where $\pi_k: \K^n \lgw\K$ is the projection to the $k$-entry, and note that all of its derivatives $\frac{\partial^{|\gamma|}}{\partial \ub^{\gamma}}\Psi_k(\cdot,0)$ are globally defined morphisms over $\Lambda$. Now consider:
\[
F(\Psi^*_{1,\pa},\ldots, \Psi^*_{n,\pa})=\sum_{\g\in\N^m}F_{\g}(\pa) \, \ub^\g
\]
where $\pa \in \Lambda$, and
\[
F_{\g}(\pa) = \sum_{i=1}^{d}  \sum_{\beta \in \N^{n-1},\,\beta \leq \gamma} M_{\g,i,\b}(\pa) A_{i,\b}+B_{\g}(\pa)
\]
with $M_{\g,i,\b}(\pa)$ and $B_{\g}(\pa)$ polynomials in the derivatives of $\Psi^*_{\pa}$. In particular, note that $M_{\g,i,\b}(\pa)$ and $B_{\g}(\pa)$ belong to $\mathcal{O}(\Lambda)$.

As before, for any $k\in\N$, we consider the finite system of linear equations:
\begin{equation}\tag{$S_{k}(\pa)$}\label{eq_S_k_par}\forall \g\in\N^m, |\g|< k,\ \ \  F_{\g}(\pa)(A_{i,\b})=0.\end{equation} 
Let $s_k$ denote the number of indexes $\g$ such that $|\g| <k$. The system \eqref{eq_S_k_par} can be written as 
\[
M^{(k)}(\pa) \cdot A^{(k)}+B^{(k)}(\pa)=0
\]
where $M^{(k)}(\pa)$ is the $(s_k \times d s_k )$-matrix with entries $M_{\g,i,\b}(\pa)$, $A^{(k)}$ is the $(ds_k\times 1)$-column with entries $A_{i,\b}$, and $B^{(k)}(\pa)$ is the $(s_k \times 1)$-column with entries $B_{\g}(\pa)$. We denote by $M_{i,\b}^{(k)}(\pa)$ the column of $M^{(k)}(\pa)$ corresponding to $A_{i,\b}$, that is:
\[
M^{(k)}(\pa) \cdot A^{(k)} = \sum_{i=1}^d \sum_{|\b| <k} M_{i,\b}^{(k)}(\pa) A_{i,\b}
\]
Let us fix $i_0\in\{1,\ldots, d\}$ and $\b_0\in\N^{n-1}$ and let us prove that there exists $\overline{a}_{i_0,\b_0} \in \LLL$ whose restriction to $U$ is equal to $a_{i_0,\b_0}$. For every $k\in \N$ with $k >|\b_0|$, let us denote by $t_{0}^{(k)}(\pa)$ the dimension of the $\K$-vector space $T^{(k)}_0(\pa)$ generated by the $M_{i,\b}^{(k)}(\pa)$  for $(i,\b)\neq (i_0,\b_0)$. There is an analytic proper subset $D^{(k)}$ of $\Lambda$ such that for every $\pa\in\Lambda\setminus D^{(k)}_0$, $t_{0}^{(k)}(\pa)$ is maximal; denote by $t^{(k)}_0$ this maximal value.

We now fix $\pa\in U\setminus  \bigcup_{k>|\beta|} D^{(k)}_0$ and consider $\mu_\pa$ the  Chevalley function of Corollary \ref{app1} associated to $f(\pa,x)$. We now fix $k=|\b_0|+1$ and we set $\ell=\mu_\pa(k)$. To simplify the notation, set $t^{(\ell)}_0 = t_0$, and consider $\K$-linearly independent vectors $M_{i_1,\b_1}^{(\ell)}(\pa)$, $M_{i_2,\b_2}^{(\ell)}(\pa)$, \ldots, $M_{i_{t_0},\b_{t_0}}^{(\ell)}(\pa)$ which generate $T^{(\ell)}_0(\pa)$. 

\begin{claim}\label{cl:1}
There exists a neighborhood $U_{\pa}$ of $\pa$ such that 
$M_{i_0,\b_0}^{(\ell)}(\pb)$ does not belong to the vector space generated by $T^{(\ell)}_0(\pb)$ for every $\pb\in U_{\pa}$.
\end{claim}
\begin{proof}
Indeed, from the definition of $T_0^{(\ell)}(\pa)$ the equality $M^{(\ell)}(\pa) \cdot A^{(\ell)}+B^{(\ell)}(\pa)=0$ can be re-written as:
\[
A_{i_0,\b_0}M_{i_0,\b_0}^{(\ell)}(\pa)+\sum_{j=1}^{t_0}(A_{i_j,\b_j}+L_{j})M_{i_j,\b_j}^{(\ell)}(\pa) +B^{(\ell)}(\pa)=0.
\]
where the $L_j$ are $\K$-linear combinations of the terms $A_{i,\b}$ with $(i,\b)\neq (i_j,\b_j)$ for $j=0,\ldots,t_0$. We recall that, by Corollary \ref{app1}, there exists a unique entry $a_{i_0,\b_0} = A_{i_0,\b_0}$ for which the above system admits a solution. It is now immediate that $M_{i_0,\b_0}^{(\ell)}(\pa) \notin T^{(\ell)}_0(\pa)$ (otherwise, for each choice of $A_{i_0,\b_0}$, it would be possible to compensate the terms $A_{i,\b}$ with $(i,\b) \neq (i_0,\b_0)$ in order to get a different solution). We conclude easily from the analyticity of the vectors $M_{i,\b}^{(\ell)}$. 
\end{proof}

Now, by analyticity of the entries $M_{i,\b}^{(\ell)}$, there is a proper analytic subset $E_0$ of $\Lambda$ such that, for every $\pb\in \Lambda\setminus E_0$, the vectors $M_{i_j,\b_j}^{(\ell)}(\pb)$, for $0\leqslant j\leqslant t_0$, are $\K$-linearly independent. Moreover, since $t_0=\max_{\pc}\{t_0^{(\ell)}(\pc)\}$, these vectors form a basis of the vector space generated by all the $M^{(\ell)}_{i,\b}(\pb)$. 
Therefore, for a given $(i,\b)\neq(i_j,\b_j)$ for $j=0,\ldots,t_0$ and for a given $\pb\in\Lambda\setminus E_0$, the equation $\sum_{j=0}^{t_0}M_{i_j,\b_j}^{(\ell)}(\pb)X_j=M_{i,\b}(\pb)$ has a unique solution $X=(X_0,\ldots, X_{t_0})\in\K$. Let us denote by $M_0(\pb)$ the $ s_{\ell} \times (t_0+1)$-matrix with columns $M_{i_j,\b_j}^{(\ell)}(\pb)$ for $j=0,\ldots,t_0$. By Cramer's rule, the $X_i$ have the form $g_i(\pb)/\D_0(\pb)$ where $g_i(\pb)$ is a minor of a matrix whose entries are some of the entries of the $M_{i_j,\b_j}^{(\ell)}(\pb)$ and of $M_{i,\b}(\pb)$, and $\D_0(\pb)$ is the determinant of  a   $(t_0+1)$-square sub-matrix $N_0(\pb)$ of $M_0(\pb)$. Therefore, there is a proper analytic subset $E_1$ of $\Lambda$, such that for every $\pb'\in\Lambda\setminus E_1$, $\D_0(\pb')\neq 0$.
In particular the system ($S_{\ell}(\pb)$), for $\pb\in\Lambda\setminus (E_0\cup E_1)$, can be rewritten as
\[
\sum_{j=0}^{t_0} M_{i_j,\b_j}^{(\ell)}(\pb)(A_{i_j,\b_j}+L_{i_j, \b_j}(\pb))+B^{(\ell)}(\pb)=0
\] 
where the $L_{i_j, \b_j}(\pb)$ are linear forms in the $A_{i,\b}$ for $(i,\b) \neq (i_j,\b_j)$ for $j=0,\ldots,t_0$, with analytic coefficients. We claim that $L_{i_0,\b_0}(\pb) \equiv 0$. Indeed, by Claim \ref{cl:1}, note that for every $\pc \in U_{\pa} \setminus  \bigcup_{k>|\beta|} D^{(k)}_0$ we have that $M_{i_0,\b_0}^{(\ell)}(\pc)$ does not belong to the $t_0$-vector space $T_0^{(\ell)}(\pc)$, implying that $ L_{i_0, \b_0}(\pc)$ is equal to zero in an open set; by analyticity $L_{i_0,\b_0}\equiv 0$. In particular the system ($S_{\ell}(\pb)$), for $\pb\in\Lambda\setminus (E_0\cup E_1)$, can be rewritten as:
\[
M_{i_0,\b_0}^{(\ell)}(\pb) A_{i_0,\b_0} + \sum_{j=1}^{t_0} M_{i_j,\b_j}^{(\ell)}(\pb)(A_{i_j,\b_j}+L_{i_j, \b_j}(\pb))+B^{(\ell)}(\pb)=0.
\] 
It now follows from Cramer's rule that there exists a solution $\overline{a}_{i_0,\b_0}(\pb)$ of the truncated system which can be expressed as a division $Q_0(\pb) / \Delta_0(\pb)$, where $Q_0(\pb)$ depends on the entries of $M_{i_j,\b_j}^{(\ell)}(\pb)$ for $j=1,\ldots,t_0\}$ and $B^{(\ell)}(\pb)$. We now remark that Claim \ref{cl:1} implies that $\overline{a}_{i_0,\b_0}(\pb) =a_{i_0,\b_0}(\pb) $ for every $\pb \in U_{\pa} \setminus (D_0 \cup Z_0)$, which implies that they are equal over $U \setminus Z_0$. We conclude that $a_{i_0,\b_0}$ can be extended as a holomorphic function on $\Lambda\setminus Z_0$ that belongs to $\LLL$. Since the choice of $(i_0,\b_0)$ was arbitrary, this proves the Lemma.
\end{proof}

\subsection{Proof of Proposition \ref{cor_extension}}

Let $\Phi: \Omega \lgw\K^n$ and $\varphi: \Lambda \lgw\Omega$ be the two morphisms from the definition of admissible family \ref{def:admissibleFamily}, and recall that $\Psi(\pa,\ub) = \Phi( \varphi(\pa) + \ub) - \Phi(\varphi(\pa))$. Let $\pa\in \RR(\Psi,\Lambda)$ and set $r:=\Gr(\Psi_\pa) = \Fr(\Psi_{\pa})$; in particular, $r = \Gr(\Phi_{\varphi(\pa)})=\Fr(\Phi_{\varphi(\pa)})$. It follows from Gabrielov's rank Theorem (or the rank Theorem \ref{thm:TemperateRank}) that $\Ar(\Phi_{\varphi(\pa)}) =r$. 

Apart from a translation in $\x$, we may suppose that $\Phi(\varphi(\pa)) =0$. Let $(Z,0)$ be the germ of analytic set defined by $\Ker(\Phi^*_{\varphi(\pa)})$ and note that $r=\dim(Z,0)$. Apart from a linear change of coordinates in $\x$, we may assume that the projection $\pi:(Z,0)\lgw (\K^r,0)$ on the first $r$ coordinates is finite. In particular, each function $x_i$, for $i>r$, is finite over the ring of convergent power series $\K\{x_1,\ldots, x_r\}$. That is, by the Weierstrass preparation theorem, there exist non zero Weierstrass polynomials 
\[
P_i(x_1,\ldots, x_r,x_{r+i})\in\K\{x_1,\ldots, x_r\}[x_{r+i}], \ \text{ for } i=1,\ldots, n-r,
\]
belonging to $\Ker(\Phi^*_{\varphi(\pa)})$. By replacing each $P_i$ by one of its irreducible factors we may assume that the $P_i$ are irreducible Weierstrass polynomials at $0$.

We claim that, apart from changing the choice of point $\pa \in \RR(\Psi,\Lambda)$ and re-centering the coordinate system $\x$ accordingly, there exists a neighborhood $U$ of $\pa$ such that $P_i$ are well-defined and irreducible at every point in $\Phi(\varphi(U))$. Indeed, let $V$ be an open neighborhood of $0$ in $\K^n$ on which the $P_i$ are well-defined, and $U$ be an open connected neighborhood of $\pa$ such that $\Phi(\varphi(U)) \subset V$. Apart from shrinking $U$ and $V$, we may suppose that $P_i \in \Ker(\Phi^*_{\varphi(\pb)}) $ for every $\pb \in U$; in particular, $U \subset \RR(\Psi,\Lambda)$. Now, recall that being not irreducible is an open property for the Euclidean topology, thus the property of being irreducible is a closed property. If one of $P_i$ is not irreducible at a point $\Phi(\varphi(\pb))$, for some $\pb \in U$, we may replace $\pa$ by $\pb$, $P_i$ by one of its irreducible factors at this point, and we shrink $U$ and $V$ accordingly. Since the degree of the $P_i$ is a positive integer, this process should end in a finite number of steps, proving the claim.

Fix $s=1$, \ldots $,n-r$, set $\x^{(s)} = (x_1,\ldots,x_r,x_{r+s})$, $\Phi^{(s)} := (\Phi_1,\ldots,\Phi_r,\Phi_{r+s})$, and denote by $\Psi^{(s)} = (\Psi_1,\ldots,\Psi_r,\Psi_{r+s})$ the  family associated to $\Phi^{(s)}$ and $\varphi$. Note that $ U \subset \RR(\Psi^{(s)},\Lambda) $ by construction. Moreover $\Ker({\Psi^{(s)}_{\pa}}^*)$ is generated by $P_s$ since $P_s$ is irreducible and   $\Ker({\Psi^{(s)}_{\pa}}^*)$ is a height one prime ideal of $\K\{\x^{(s)}\}$. We set: 
\[
f_s(\pb,\x^{(s)}):=P_i\left(\Phi^{(s)}( \varphi(\pb)) + \x^{(s)}\right)
\]
for every $\pb\in U$, which can be written as:
\[
f_s(\pb,\x^{(s)}) = y^d + a_1(\pb,\x')y^{d-1} + \cdot + a_d(\pb,\x')
\]
where $y = x_{r+s}$ and $\x' = (x_1,\ldots,x_r)$. First, note that $a_i(\pb,\x') \in \mathcal{O}(U) \lb \x' \rb$ since $\Phi \circ \varphi$ is an analytic map defined on $U$ and $P_i$ is well defined in $\Phi(\varphi(U))$. Second, note that $f_s(\pb,\x^{(s)}) \in \Ker({\Psi^{(s)}_{\pb}}^*)$ for every $\pb \in U$ since ${\Psi^{(s)}_{\pb}}^*(f_i) = P_s \circ \Phi^{(s)}_{\varphi(\pb)}  \equiv 0$, and that $f_s(\pb,\x^{(s)})$ generates $\Ker({\Psi^{(s)}_{\pb}}^*)$ since $P_i$ is irreducible. Third, note that $f_i(\pb,0,y) = y^{k(\pb)} U(\pb,y)$ for some $1 \leq k(\pb) \leq d$ and $U(\pb,y)$ is a monic polynomial in $y$ coprime with $y$. By Hensel Lemma (see \cite[18.5.13]{Gr}), this implies  that $f_i(\pb,\x',y)$ is the product of two monic polynomials of degree $k(\pb)$ and $d-k(\pb)$ respectively.
From the fact that $P_i$ is irreducible and $k(\pb)>0$ at every point $\pb \in U$, we conclude that $k(\pb)=d$, that is, $f_s(\pb,0,y) = y^d$. These three observations show that $f_s$ satisfies all hypothesis of Lemma \ref{cor_linear_equations}, so that it can be extended as a power series $\overline{f}_s(\x^{(s)})$ of $\LLL\lb \x\rb$, where $\LLL$ is the fraction field of $\mathcal{O}(\Lambda)$, such that $\Psi_{\LLL}^{\ast}(\overline{f}_s) =0$. We conclude that $\Fr(\Psi_\LLL^*)\leq r$, and since $\Gr(\Psi_\LLL^*)=r$, we get that $\Gr(\Psi)=\Fr(\Psi_\LLL^*)$, finishing the proof.





\begin{thebibliography}{00}


\bibitem[ABM08]{ABM} J. Adamus, E.  Bierstone, P. D.  Milman, Uniform linear bound in Chevalley's lemma, \emph{Can. J. Math.}, \textbf{60}, No. 4, (2008), 721-733.

\bibitem[BCR21]{BCR} A. Belotto da Silva, O. Curmi, G. Rond, A proof of Gabrielov's rank Theorem, \emph{J. \'Ec. Polytech., Math.}, \textbf{8}, (2021), 1329-1396. 

\bibitem[BCR22]{BCR2} A. Belotto da Silva, O. Curmi, G. Rond, On rank Theorems for morphisms of local rings, preprint, 2022.


\bibitem[BM87]{BMFourier} E. Bierstone and P. Milman, Relations among analytic functions 8, \emph{Ann. Inst. Fourier},
\textbf{37}, No. 1, 187-239, (1987). 

\bibitem[BM88]{BMihes} E. Bierstone and P. Milman,  Semianalytic and subanalytic sets, \emph{Publ. Math., Inst. Hautes \'Etud. Sci.}, \textbf{67}, (1988), 5-42. 

\bibitem[BM00]{BMsubanalytic} E. Bierstone and P. Milman, Subanalytic Geometry, Model Theory, Algebra, and Geometry. MSRI Publications. Volume 39, (2000).



\bibitem[Ch43]{Ch} C. Chevalley, On the theory of local rings, \emph{Ann. of Math.}, \textbf{44}, (1943), 690-708.

\bibitem[Da74]{Da} H. G.  Dales, The ring of holomorphic functions on a Stein compact set as a unique factorization domain, \emph{Proc. Amer. Math. Soc.}, \textbf{44}, (1974), 88-92.


\bibitem[DvdD88]{DVdD} J. Denef, L. van den Dries,  $p$-adic and real subanalytic sets,  \emph{Ann. of Math. (2)}, \textbf{128}, (1988), no. 1, 79--138.



\bibitem[FG85]{FG1} E. Fortuna, M. Galbiati,  Quelques r\'esultats sur les ensembles Nash sous-analytiques,
\emph{Bull. Soc. Math. Fr.}, \textbf{113}, (1985), 347-358. 

\bibitem[FG86]{FG2} E. Fortuna, M. Galbiati, Semi-analyticit\'e et sous-analyticit\'e,  \emph{Ann. Mat. Pura Appl., IV. Ser.}, \textbf{143}, (1986), 363-372. 

\bibitem[Fr67]{Fr} J. Frisch, Points de platitude d'un morphisme d'espaces analytiques complexes, \emph{Invent. Math.}, \textbf{4}, (1967), 118-138. 


\bibitem[Ga71]{Ga1} A. M. Gabrielov, The formal relations between analytic functions, \emph{Funkcional. Anal. i Prilovzen}, \textbf{5}, (1971), 64-65.


\bibitem[Gro67]{Gr} A. Grothendieck, \'El\'ements de G\'eom\'etrie alg\'ebrique, \emph{Publ. Math. IHES}, \textbf{32}, (1967).

\bibitem[GuRo65]{GR} R.C. Gunning, H. Rossi, "Analytic functions of several complex variables", Prentice-Hall (1965)

\bibitem[Hi73]{HirPisa} H. Hironaka, Introduction to real-analytic sets and real-analytic maps, Istituto Matematico ``L. Tonelli", Pisa, 1973.

\bibitem[Hi79]{HirEqui} H. Hironaka,
\emph{On Zariski Dimensionality Type},
American Journal of Mathematics, Vol. 101, No. 2 (1979), pp. 384-419.


\bibitem[Hi86]{Hir3} H. Hironaka, Local analytic dimensions of a subanalytic set, Proc. Japan Acad. Ser. A Math. Sci. 62 (1986), 73--75.

\bibitem[Ho88]{Hor} L. H\"ormander, \textit{An Introduction to Complex Analysis in Several Variables}, (3rd ed.), North Holland, ISBN 978-1-493-30273-4, 1988 [1966].



\bibitem[Ku88]{Ku1} K.  Kurdyka, Points r\'eguliers d'un sous-analytique, \emph{Ann. Inst. Fourier}, \textbf{38}, (1988), no. 1, 133-156.


\bibitem[\L o65]{Loj} S. {\L}ojasiewicz,
\textit{Ensembles semialg\'ebriques}, Inst. Hautes Etudes Sci., Bures-sur-Yvette, 1965; notes of lectures at Universit\'e d'Orsay, 
re-edited by M. Coste, 2006, \url{https://perso.univ-rennes1.fr/michel.coste/Lojasiewicz.pdf}.

\bibitem[\L o93]{LojCite} S. {\L}ojasiewicz, \textit{Sur la géométrie semi- et sous-analytique}, Ann. Inst. Fourier (Grenoble) 43 (1993), no. 5, 1575–1595


\bibitem[Mi78]{Mil1} P. Milman, Complex analytic and formal solutions of real analytic equations in $\mathbb{C}$, \emph{Math. Ann.} 233, (1978), 1-7.

\bibitem[Na66]{n66} R. Narasimhan,
{\it Introduction to the theory of analytic spaces}.
Lect. notes in math., 25, Springer, 1966.



\bibitem[Os1916]{Os} W. F. Osgood, On functions of several complex variables, \textit{Trans. Amer. Math. Soc.}, \textbf{17}, (1916), 1-8.



\bibitem[PP21]{PP21} A. Parusinski, L. Paunescu, \emph{Zariski's dimensionality type of singularities. case of dimensionality type 2},
preprint, arXiv:2104.07156 [math.AG], 2021.


\bibitem[Pa90]{Pawthesis} W. Paw\l ucki, Points de Nash Des Ensembles Sous-analytiques. Memoirs of the American Mathematical Society, Volume 84, Number 425, 1990, 76 pages.

\bibitem[Pa92]{Paw} W. Paw\l ucki, On Gabrielov's regularity condition for analytic mappings, \emph{Duke Math. J.}, \textbf{65}, No. 2, (1992), 299-311. 

\bibitem[Pa89]{Paw2} W. Paw\l ucki, On relations among analytic functions and geometry of subanalytic sets, \emph{Bull. Polish Acad. Sci. Math.}, \textbf{37}, (1989), 117-125.

\bibitem[Ta81]{Tamm} M. Tamm, Subanalytic sets in the calculus of variation, \emph{Acta Math.}, \textbf{146}, (1981), 167-199.


\bibitem[To90]{To2} J.-Cl. Tougeron, Sur les racines d'un polyn\^ome \`a coefficients s\'eries formelles, \emph{Real analytic and algebraic geometry (Trento 1988)}, 325-363, \emph{Lectures Notes in Math.}, \textbf{1420}, (1990).

\bibitem[Za79]{ZarEisenstein} O. Zariski,
\emph{Foundations of a General Theory of Equisingularity on $r$-Dimensional Algebroid and Algebraic Varieties, of Embedding Dimension $r + 1$}, 
American Journal of Mathematics, Vol. 101, No. 2 (1979), pp. 453-51.

\end{thebibliography}
\end{document}